\documentclass[a4paper,12pt,reqno]{amsart}
 
\usepackage[margin=2cm]{geometry} 
\usepackage{amsmath,amsthm,amssymb,enumitem}
\usepackage{amsrefs}
\usepackage[ocgcolorlinks,hyperfootnotes=false,colorlinks=true,citecolor=blue,linkcolor=blue,urlcolor=blue]{hyperref}

\newcommand{\fl}{f{\kern0.075em}l}
\newcommand{\norm}[1]{\left\lVert #1 \right\rVert}

\theoremstyle{plain}
\newtheorem{thm}{Theorem}[section]
\newtheorem{lem}[thm]{Lemma}
\newtheorem{prop}[thm]{Proposition}
\newtheorem{cor}[thm]{Corollary}

\theoremstyle{definition}
\newtheorem{defn}[thm]{Definition}
\newtheorem{exmp}[thm]{Example}

\setlist[enumerate]{itemsep=2mm, topsep=0mm}

\title[Asymptotic behaviour of the Bergman invariant and Kobayashi metric]{Asymptotic behaviour of the Bergman invariant and Kobayashi metric on exponentially flat infinite type domains
}
\author{Ravi Shankar Jaiswal}
\address{Centre for Applicable Mathematics, Tata Institute of Fundamental Research, Bangalore 560065, India.}
\email{ravi@tifrbng.res.in}
\date{\today}

\begin{document}

\addtolength{\jot}{2mm}
\addtolength{\abovedisplayskip}{1mm}
\addtolength{\belowdisplayskip}{1mm}

\maketitle
\begin{abstract}
We prove the nontangential asymptotic limits of the Bergman canonical invariant, Ricci and Scalar curvatures of the Bergman metric, as well as the Kobayashi--Fuks metric, at exponentially flat infinite type boundary points of smooth bounded pseudoconvex domains in $\mathbb{C}^{n + 1}, \, n \in \mathbb{N}$. Additionally, we establish the nontangential asymptotic limit of the Kobayashi metric at exponentially flat infinite type boundary points of smooth bounded domains in $\mathbb{C}^{n + 1}, \, n \in \mathbb{N}$. We first show that these objects satisfy appropriate localizations and then utilize the method of scaling to complete the proofs.
\end{abstract}
\section{Introduction}\label{Intro}
The Bergman canonical invariant, Ricci and Scalar curvatures of the Bergman metric, and the Kobayashi and Kobayashi--Fuks metrics are biholomorphic invariants. Studying the boundary behaviour of biholomorphic invariant objects is crucial in complex geometry, as it helps analyze the regularity and extendability of holomorphic mappings near the boundary of domains.

Let $D \subset \mathbb{C}^n$ be a domain. Consider the Dirichlet problem for the complex Monge-Ampère equation,
\begin{align}\label{M-A align}
\begin{cases}\vspace{3mm}
\mathcal{J}(u) := (-1)^n \operatorname{det}\begin{pmatrix}
u & u_{\bar{k}}\\
u_j & u_{j\bar{k}}
\end{pmatrix}_{1 \leq j, k \leq n}
&= 1 \quad \text{in} \, \, D, \text{ and}\\
\quad \quad \quad \quad \quad \quad \quad \quad \quad \quad u &= 0 \quad \text{on} \, \, bD.
\end{cases}
\end{align}
In the case of the unit ball $B_n(0, 1)$ in 
$\mathbb{C}^n$, we see that the Bergman kernel on
the diagonal is given by $\kappa_{B_n(0, 1)}(z) = (1 - |z|^2)^{-(n + 1)}/c_{n}$ while $u(z) = 1 - |z|^2$ is a solution of \eqref{M-A align}. For strongly pseudoconvex domain $D$, Hörmander \cite{Hörmander 1965} shows that $\mathcal{J}\left((c_n \kappa_D)^{-1/(n + 1)}\right) = 1$ on $bD$.
In 1976, Fefferman \cite{Fefferman 1976} posed the following question: how closely does the Bergman kernel on the diagonal $\kappa_D$ approximate $c/u^{n + 1}$, where $u$ is a solution of \eqref{M-A align}, and $c$ is a positive constant? Alternatively, is $\mathcal{J}\left(c \, (\kappa_{D})^{-1/(n + 1)}\right)$ nearly equal to $1$? More precisely, does $\mathcal{J}\left(c \, (\kappa_{D})^{-1/(n + 1)}\right)$ asymptotically approach $1$ on $bD$? 
Interestingly, Ramadanov \cite{Ramadanov 1984} showed that
\begin{align}\label{2}
    \mathcal{J}\left[\kappa_D^{-\frac{1}{n + 1}}\right] = (n + 1)^{-n}J_D,
\end{align}
where $J_D$ is the Bergman canonical invariant.
Therefore, by utilizing \eqref{2}, it becomes interesting to investigate the boundary behaviour of the Bergman canonical invariant $J_D$.

The Bergman canonical invariant $J_D$ was initially introduced by Bergman \cite{Bergman 1970} for defining a general Kähler metric. 
It was Bergman \cite{Bergman 1970} who initiated the study of the boundary behaviour of this invariant on certain special classes of domains in $\mathbb{C}^2$ (strongly pseudoconvex domains, domains with some star symmetries, etc.) and found the precise nontangential limits. For a smooth strongly pseudoconvex domain $D$, $\lim_{z \to bD}J_{D}(z) = (n + 1)^n \pi^n / n!$ (see Hachaichi \cite{Hachaichi} and Krantz-Yu \cite{Krantz-Yu 1996}). Krantz and Yu \cite{Krantz-Yu 1996} showed the nontangential boundary limits of this invariant on a large class of weakly pseudoconvex bounded domains in $\mathbb{C}^{n}$ of finite type in the sense of D'Angelo \cite{D'Angelo 1982}.

Kobayashi \cite{Kobayashi} demonstrated that the Ricci curvature of the Bergman metric is bounded above by $n + 1$ for complex manifolds possessing nontrivial Bergman kernel and metric. This result was independently verified by Fuks \cite{Fuks} for bounded domains in $\mathbb{C}^n$, with the sharpness of the upper bound noted by Nazarjan \cite{Nazarjan}. Fuks \cite{Fuks} also studied the boundary behaviour of the Ricci curvature of the Bergman metric, particularly for certain strongly pseudoconvex domains in $\mathbb{C}^2$. Krantz and Yu \cite{Krantz-Yu 1996} established the nontangential boundary limits of the Ricci and Scalar curvatures of the Bergman metric on a large class of weakly pseudoconvex bounded domains in $\mathbb{C}^{n}$ of D'Angelo finite type.

Given that the Ricci curvature of the Bergman metric on a bounded domain in $\mathbb{C}^n$ is strictly bounded above by $n + 1$, the Kobayashi--Fuks metric, introduced in Subsection \ref{2.2}, comes into focus. Borah and Kar \cite{Kar} investigated its boundary behaviour on strongly pseudoconvex domains in $\mathbb{C}^n$. Additionally, Kar \cite{Kar 2023} extended the analysis to explore the nontangential boundary limits of the Kobayashi--Fuks metric on a large class of weakly pseudoconvex bounded domains in $\mathbb{C}^n$ of D'Angelo finite type.

The nontangential asymptotic boundary behaviour of the Kobayashi metric for strongly pseudoconvex domains was initially analyzed by Graham \cite{Graham 1975} in 1975. Subsequently, in 1995, J. Yu \cite{Yu 1995 2} investigated the nontangential asymptotic behaviour of the Kobayashi metric on a broad class of weakly pseudoconvex bounded domains in $\mathbb{C}^{n}$ of D'Angelo finite type.

The main objective of this article is to establish the asymptotic behaviour of the Bergman canonical invariant, Ricci and Scalar curvatures of the Bergman metric, and the Kobayashi and Kobayashi--Fuks metrics at an infinite type boundary point of a smooth bounded domain in $\mathbb{C}^{n + 1}, \, n \in \mathbb{N}$.

Let $D \subset \mathbb{C}^{n + 1}$ be a bounded smooth domain with $0 \in bD$. The boundary point
$0$ is said to be \emph{exponentially flat} if there exists a local defining function of D near the origin of the form
\begin{align}\label{1}
    \rho(z_1, \dots, z_{n + 1}) = \operatorname{Re}z_1 + \phi \left(|z_2|^2 + \dots + |z_{n + 1}|^2\right),
\end{align}
 where $\phi : \mathbb{R} \to \mathbb{R}$ is a smooth function satisfying the Definition \ref{exp flat fun}.
\begin{defn}\label{exp flat fun}
    A smooth function $\phi: \mathbb{R} \to \mathbb{R}$ is said to be
    \emph{exponentially flat near the origin}, if it satisfies the following properties:
    \begin{enumerate}
    \item $\phi(x) = 0,$ for $x \leq 0$, 
    \item $\phi^{(m)}(0) = 0$, for $m \in \mathbb{N}$,
    \item there exists $\epsilon_0 > 0$ such that $\phi''(x) > 0$, for $0 < x < \epsilon_0$,
    \item the function $-1/\operatorname{log}\phi(x)$ extends to a smooth function on $[0,\infty)$ and vanishes to a finite order $m$ at $0$, and
    \item \begin{align}\label{Scaling}
        \lim_{r \to 0^{+}}\frac{\phi(rx)}{\phi(r)} = \begin{cases} 
      0,  &\text{if } 0 < x < 1, \text{ and}\\
      \infty,  & \text{if } x > 1.
      \end{cases}
    \end{align}
\end{enumerate}
\end{defn}
\begin{exmp}
For $m \in \mathbb{N}$,
\begin{align}
    \phi(x) = \begin{cases} 
      0,  &\text{if } x \leq 0, \text{ and}\\
      \operatorname{exp}\left(-1/x^m\right),  & \text{if } x > 0.
      \end{cases}
\end{align}
is exponentially flat near the origin.
\end{exmp}
We will now introduce the essential definitions and notations needed to state our results. We utilize the following generalized cones for nontangential approach. For $\alpha, N > 0$, we define
\begin{align}
    C_{\alpha,N} := \Big\{\big(z_1,z'\big) \in \mathbb{C}\times\mathbb{C}^{n} : \operatorname{Re}z_1 < -\alpha|z'|^N\Big\},
\end{align}
where $z' = (z_2, \dots, z_{n + 1})$. A smooth curve $q: (0, \epsilon_0) \to D \cap C_{\alpha, N}$ is said to be $(\alpha, N)$-cone type curve approaching $0 \in bD$ if
\begin{align*}
    \lim_{t \to 0^+} q(t) = (0, \dots, 0),
\end{align*}
for some $\epsilon_0 > 0$.
Let $\xi \in \mathbb{C}^{n + 1}$. We denote the complex normal and complex tangential components of $\xi$ with respect to $\pi(q(t))$ as $\xi_{N,t}$ and $\xi_{T,t}$, respectively. Here, $\pi: \mathbb{C}^n \to \partial D$ represents the normal projection onto $\partial D$ for points close to $\partial D$. Therefore, $\xi$ can be expressed as the sum $\xi = \xi_{N,t} + \xi_{T,t}$, where $\xi_{T,t} \in T_{\pi(q(t))}^{\mathbb{C}}(\partial D)$ and $\xi_{N,t}$ is orthogonal to $T_{\pi(q(t))}^{\mathbb{C}}(\partial D)$. Let $d(t)$ and $d^*(t)$ denote the Euclidean distances from $q(t)$ to $\partial D$ along the normal and tangential directions, respectively.

We pove the boundary behaviour of the Bergman canonical invariant $J_D$, Ricci curvature $R_D$ and Scalar curvature $S_D$ on exponentially flat infinite type pseudoconvex domains in the following theorem.
\begin{thm}\label{curvatures}
   Let $D \subset \mathbb{C}^{n + 1}$ be a smooth bounded pseudoconvex domain with $0 \in bD$. If $0$ is an exponentially flat boundary point and $q$ is an $(\alpha, N)$-cone type curve approaching $0$. Then, for $\xi \in \mathbb{C}^{n + 1} \setminus \{0\}$, we have
    \begin{enumerate}[labelsep=10mm,leftmargin=20mm,itemsep=5mm]
    \item $\begin{aligned}[t]
    \lim_{t \to 0^{+}} {J_{D}(q(t))} = \frac{2\pi^{n + 1} (n + 1)^n}{n!},
    \end{aligned}
    $
    \item $\begin{aligned}[t]
        \lim_{t \to 0^{+}} {R_{D}(q(t) ; \xi)} = -1, \text{ and}
    \end{aligned}$
    \item 
    $\begin{aligned}[t]
    \lim_{t \to 0^{+}} {S_{D}(q(t))} = -(n + 1).
    \end{aligned}$
    \end{enumerate}
\end{thm}
We prove the boundary behaviour of the Kobayashi--Fuks metric, denoted by $M_D^{F}$, on exponentially flat infinite type pseudoconvex domains in the following theorem.
\begin{thm}\label{1.4}
   Let $D \subset \mathbb{C}^{n + 1}$ be a smooth bounded pseudoconvex domain with $0 \in bD$. If $0$ is an exponentially flat boundary point and $q$ is an $(\alpha, N)$-cone type curve approaching $0$. Then, for $\xi \in \mathbb{C}^{n + 1} \setminus \{0\}$, we have 
    \begin{align}
        \lim_{t \to 0^{+}}\frac{M_D^{F}(q(t); \xi)}{{\sqrt{{|\xi_{N,t}|^2}/{2d(t)^2} + (n+1) {|\xi_{T,t}|^2}/{d^*(t)^2}}}} = {(n + 3)^{1/2}}.
    \end{align}
\end{thm}
We also prove the boundary behaviour of the Kobayashi metric, denoted by $M_D^{K}$, on exponentially flat infinite type domains in the following theorem.
\begin{thm}\label{1.5}
     Let $D \subset \mathbb{C}^{n + 1}$ be a smooth bounded pseudoconvex domain with $0 \in bD$. If $0$ is an exponentially flat boundary point and $q$ is an $(\alpha, N)$-cone type curve approaching $0$. Then, for $\xi \in \mathbb{C}^{n + 1} \setminus \{0\}$, we have
     \begin{align}
          \lim_{t \to 0^+} \frac{M_D^K(q(t); \xi)}{\operatorname{max}\left\{{|{{\xi}}_{N,t}|}/{2d(t)},{|{{\xi}}_{T,t}|}/{d^*(t)}\right\}} = 1.
     \end{align}
\end{thm}
The above result, in $n = 1$, is stated in Lee \cite{Sunhong 2001}. However, \cite{Sunhong 2001}*{equation $(14)$}
is not true and it is used crucially in their proofs. We have explicitly highlighted this error in an earlier article (see \cite{Ravi}). Our results provide an alternate proof for some of the results in Lee \cite{Sunhong 2001}.

We now state the boundary limits of Kobayashi and Kobayashi--Fuks metrics in the normal direction at exponentially flat infinite type boundary points of domains in $\mathbb{C}^{n + 1}, \, n \in \mathbb{N}$.
\begin{cor} Let $D \subset \mathbb{C}^{n + 1}$ be a smooth bounded domain with $0 \in bD$. If $0$ is an exponentially flat boundary point. Then, for $\xi \in \mathbb{C}^{n + 1} \setminus \{0\}$, we have
 \begin{enumerate}[labelsep=2mm,leftmargin=10mm,itemsep=5mm]
    \item \quad \quad \quad \quad \quad \quad \, 
    $\begin{aligned}
          \lim_{t \to 0^+} \frac{M_D^K((-t, 0); \xi)}{\operatorname{max}\left\{{|{{\xi}}_{1}|}/{2d(t)},{|{{\xi'}}|}/{[\phi^{-1}(t)]^{1/2}}\right\}} = 1.
    \end{aligned}
    $
    \item Further assume that $D$ is
    pseudonvex domain. Then, for $\xi \in \mathbb{C}^{n + 1} \setminus \{0\}$, we have
    \begin{align*}
        \lim_{t \to 0^{+}}\frac{M_D^{F}((-t, 0); \xi)}{{\sqrt{{|\xi_{1}|^2}/{2d(t)^2} + (n+1) {|\xi'|^2}/{\phi^{-1}(t)}}}} = {(n + 3)^{1/2}}.
    \end{align*}
    \end{enumerate}  
\end{cor}
The structure of the article is as follows: Section \ref{Prelim} focuses on definitions, known results, and preliminary findings crucial for proving our main results. Section \ref{Cone} is subdivided into two subsections. In the first subsection, we establish the boundary limits of the Bergman canonical invariant, Ricci and Scalar curvatures of the Bergman metric, as well as the Kobayashi--Fuks metric. The second subsection is dedicated to proving the boundary limit of the Kobayashi metric.
\section{Preliminaries}\label{Prelim}
\subsection{Bergman and Kobayashi--Fuks metrics}\label{2.2}
Consider a bounded domain $D \subset \mathbb{C}^n$. 
The Bergman space of $D$, denoted by $A^2(D)$, consists of holomorphic functions $f: D \to \mathbb{C}$ that belong to $L^2(D)$. Since $A^2(D)$ is a closed subspace of $L^2(D)$, there exists an orthogonal projection from $L^2(D)$ onto $A^2(D)$, known as the Bergman projection of $D$, denoted by $P_{D}$.
This projection can be expressed by an integral kernel called the Bergman kernel and denoted by $K_D: D \times D \to \mathbb{C}$. That is,
\begin{align}
    P_{D}f(z) = \int_{D}K_{D}(z, w)f(w) \, \mathrm{d}V(w) \quad \text{for all } f \in L^2(D), z \in D.
\end{align}

The Bergman kernel possesses the following properties:
\begin{enumerate}
    \item For a fixed $w \in D$, $K_D(\cdot, w)\in A^2(D)$,
    \item $\overline{K_D(z, w)} = K_D(w,z)$ for all $z,w\in D$, and
    \item $f(z)=\int_D K_D(z, w)f(w)\,\mathrm{d}V(w)$ for all $f\in A^2(D)$ and $z\in D$.
\end{enumerate} 

The Bergman kernel on the diagonal of $D$ is given by
\begin{align}
    \kappa_{D}(z) = K_D(z, z), \quad z \in D.
\end{align}
Additionally, the Bergman metric of $D$ can be defined as 
\begin{align} 
B_D(z;\xi)=\sqrt{\sum_{i,j=1}^{n}g_{i\bar{j}}(z)\xi_i\overline{\xi}_j},
\end{align} where 
$g_{i\bar{j}}(z) = \partial^2_{z_i\bar{z}_j} \operatorname{log} \kappa_D(z)$, $z \in D$, and $\xi \in \mathbb{C}^n$.

The Ricci curvature for the Bergman metric is given by
\begin{align}
     R_{D}(z; \xi) = \frac{\sum_{h,j,k,l}g^{k\bar{l}}(z)R_{\bar{h}jk\bar{l}}(z)\bar{\xi}_h \xi_j}{B_{D}(z;\xi)^2}, \quad z \in D \, \text{and } \xi \in \mathbb{C}^{n} \setminus \{0\}.
\end{align}
In the above equation, $R_{\bar{h}jk\bar{l}}(z) = -\partial^2_{z_k\bar{z}_l}g_{j\bar{h}}(z) + \sum_{\mu, \nu} g^{\nu \bar{\mu}}(z) \partial_{z_k}g_{j\bar{\mu}}(z) \partial_{\bar{z}_l}g_{\nu \bar{h}}(z) ,$ where $[g^{\nu \bar{\mu}}(z)]$ is the inverse matrix of $[g_{j \bar{k}}(z)]$.

The scalar curvature of the Bergman metric is given by
\begin{align}
    S_D(z) = \sum_{{h}jk{l}}g^{j\bar{h}}(z)g^{k\bar{l}}(z)R_{\bar{h}jk\bar{l}}(z), \quad z \in D.
\end{align}
The Bergman canonical invariant is given by
\begin{align}\label{Bergman invariant}
    J_D(z) = \frac{\operatorname{det}G_D(z)}{\kappa_D(z)},
\end{align}
where $G_D(z) = [g_{j \bar{k}}(z)],\, z \in D$.

The Ricci cuvature of the Bergman metric on a bounded domain in $\mathbb{C}^{n}$ is strictly bounded above by $n + 1$ (see Kobayashi \cite{Kobayashi}, Fuks \cite{Fuks} and Nazarjan \cite{Nazarjan}). Therefore, we can define the Kobayashi--Fuks metric of $D$ as 
\begin{align}
    M_D^{F}(z; \xi) = B_{D}(z; \xi)\sqrt{(n + 1) - R_D(z; \xi)}, \quad z \in D \text{ and } \xi \in \mathbb{C}^{n}.
\end{align}
We introduce the following extremal functions, which play a crucial role in the computation and estimation of $J_D, R_D, S_D,$ and $M_D^{F}$.
\begin{align}
    \lambda_{D}^k(z) &= \operatorname{sup}\left\{\left|\frac{\partial f}{\partial z_k}(z)\right|^2: f \in S_z, \, \frac{\partial f}{\partial z_j}(z) = 0, \text{ for } 1 \leq j < k\right\},\\
    I_{D}(z; \xi) &= \operatorname{sup}\left\{\xi f''(z)\overline{G}_{D}^{-1}(z)\overline{f''(z)}\xi^{*} :
    f \in S_z, \, f'(z) = 0\right\}, \\
    M_{D}(z; \xi) &= \operatorname{sup}\left\{\kappa_{D}^{n -1}(z) \xi f''(z)\, \overline{\operatorname{{ad}}G_{D}(z)} \, \overline{f''(z)} \xi^{*}: f \in S_z, \, f'(z) = 0\right\},\\
    L_{D}(z) &= \operatorname{sup}\left\{\operatorname{tr}\left(f''(z)\overline{G}_{D}^{-1}(z) \overline{f''(z)} G_{D}^{-1}(z)\right) : f \in S_z, \, f'(z) = 0\right\},\\
    N_{D}(z) &= \operatorname{sup} \left\{\kappa_{D}^{2n - 2}(z)\operatorname{tr}\left(f''(z)\, \overline{\operatorname{ad}G_{D}(z)} \,  \overline{f''(z)}\, G_{D}(z)\right): f \in S_z, \, f'(z) = 0\right\}.
\end{align}
Here $S_z = \{f \in A^2(D): \norm{f}_{L^2(D)} = 1, \, f(z) = 0 \}, \, z \in D, \, \xi \in \mathbb{C}^n, \, \xi^{*}$ is the tanspose conjugate vector of $\xi$, $f'$ denotes the gradient of $f$, $\operatorname{ad}G$ is the adjoint matrix whose $(ij)-$th entry is the cofactor of $g_{j \bar{i}}$, $\operatorname{tr}(\cdot)$ stands for the trace of the matrix, and $f'' = \left(\frac{\partial^2 f}{\partial z_i \partial z_j}\right)_{n \times n}$.

The following proposition gives a useful representation of the Bergman canonical invariant, Ricci curvature and Scalar curvature in terms of extremal functions.
\begin{prop}\label{Pro 2.1}(Krantz and Yu \cite{Krantz-Yu 1996})\label{2.4} Let $D$ be a bounded domain in $\mathbb{C}^n$. Then, for $z \in D$ and $\xi \in \mathbb{C}^n\setminus\{0\}$, we have
\begin{enumerate}
    \item $J_{D}(z) = \frac{\lambda_{D}(z)}{\kappa_{D}^{n + 1}(z)}$, where $\lambda_{D}(z) = \lambda_{D}^1(z) \dots \lambda_{D}^n(z)$.
    \item $R_{D}(z; \xi) = (n + 1) - \frac{1}{B_{D}^2(z; \xi)\kappa_{D}(z)}I_{D}(z; \xi)$.
    \item $R_{D}(z; \xi) = (n + 1) - \frac{1}{B_{D}^2(z; \xi)\kappa_{D}^{n + 1}(z) J_{D}(z)}M_{D}(z; \xi)$.
    \item $S_{D}(z) = n(n + 1) - \frac{1}{\kappa_{D}(z)}L_{D}(z) = n(n + 1) - \frac{1}{\kappa_{D}^{2n + 1}(z)J_{D}(z)}N_{D}(z)$.
\end{enumerate}
\end{prop}
If $f: D_1 \to D_2$ is a biholomorphism, it is easy to observe, based on Proposition \ref{Pro 2.1}, that the following transformations hold.
\begin{align}
      \lambda_{D_1}(z) &= |\operatorname{det}J_{\mathbb{C}}f(z)
        |^{2(n + 1)} \lambda_{D_2}(f(z)),\\
        I_{D_1}(z; \xi) &= I_{D_2}(f(z); f'(z)\xi)|\operatorname{det}J_{\mathbb{C}}f(z)
        |^{2},\\
        M_{D_1}(z; \xi) &= M_{D_2}(f(z); f'(z)\xi)|\operatorname{det}J_{\mathbb{C}}f(z)
        |^{2(n + 1)}, \text{ and}\\
        N_{D_1}(z) &= N_{D_2}(f(z)) |\operatorname{det}J_{\mathbb{C}}f(z)
        |^{2(2n + 1)}, 
\end{align}
where $z \in D$ and $\xi \in \mathbb{C}^n$.

We use the following proposition to prove localization for extremal functions (Lemma \ref{localisation lemma 4}).
\begin{prop}[Krantz, Yu \cite{Krantz-Yu 1996}]\label{Krantz-Yu}  
\begin{enumerate}
[labelsep=2mm,leftmargin=7mm,itemsep=5mm]
    \item Let $D_2 \subset D_1$ be domains in $\mathbb{C}^n$. For any $z \in D_2,$ $v \in \mathbb{C}^n$ and any $n \times n$ matrix $A$, we have
        \begin{align}
                v\overline{G}_{D_1}^{-1}(z)v^{*} &\geq \frac{\kappa_{D_1(z)}}{\kappa_{D_2}(z)}v\overline{G}_{D_2}^{-1}(z)v^{*}\\
                \operatorname{tr}\left(A\overline{G}_{D_1}^{-1}(z)A^{*}G_{D_1}^{-1}(z)\right) &\geq \frac{\kappa_{D_1}^2(z)}{\kappa_{D_2}^2(z)} \operatorname{tr}\left(A\overline{G}_{D_1}^{-1}(z)A^{*}G_{D_1}^{-1}(z)\right).
        \end{align}
        \item The functions $\lambda_{D}(z), M_{D}(z; \xi)$ and $N_{D}(z)$ are monotonically decreasing with respect to $D$.
\end{enumerate}
\end{prop}
The following formulas for the Bergman canonical invariant, Ricci and Scalar curvatures of the Bergman metric, and the Kobayashi-Fuks metric will be required to establish our results.
\begin{lem}\label{Formula for kernel} The Bergman canonical invariant, Ricci and Scalar curvatures of the Bergman metric, and Kobayashi--Fuks metric on $\mathbb{D} \times B_n(0, 1)$ at $0$ are as follows.
\begin{align}
    &J_{\mathbb{D}\times B_n(0, 1)}(0) = \frac{2\pi^{n + 1}(n + 1)^n}{n!}, \\
    &R_{\mathbb{D} \times B_n(0, 1)}(0; \xi) = -1, \\
    &S_{\mathbb{D} \times B_n(0, 1)}(0) = - (n + 1), \text{ and} \\
    &M_{\mathbb{D} \times B_n(0, 1)}^F(0; \xi) = (n + 3)^{1/2}(2 |\xi_1|^2 + (n + 1)|\xi'|^2)^{1/2},
\end{align}
for $\xi = (\xi_1, \xi') \in \mathbb{C} \times \mathbb{C}^{n} \setminus \{0\}$.
\end{lem}
\begin{proof}
    By using \cite{Jarnicki}*{Example $6.1.5$ and Theorem $6.1.11$}, we get
    \begin{align*}
   K_{\mathbb{D} \times B_n(0, 1)}((z_1, z'),(w_1, w')) &= K_{\mathbb{D}}(z_1, w_1) \cdot K_{B_n(0, 1)}(z', w')\\
   &= \frac{n!}{\pi^{n + 1}}(1 - z_1\bar{w}_1)^{-2}\left(1 - \sum_{j = 2}^{n + 1}z_j\bar{w}_j\right)^{-(n + 1)},     \end{align*}
    for all $z_1, w_1 \in \mathbb{D}$ and $z', w' \in B_n(0, 1)$.

    By using the above formula, we get the Bergman metric for the product domain $\mathbb{D} \times B_n(0,1)$.
    \begin{align*}
        g_{j\bar{k}}(z_1, z') = \begin{cases} 
      \frac{2}{\left(1 - |z_1|^2\right)^2}, &\text{if } j = k = 1,\\[2mm]
      \frac{n  + 1}{\left(1 - |z'|^2\right)^2}\left[\left(1 - |z'|^2\right)\delta_{jk} + \bar{z}_j z_k\right],&\text{if } j, k \in \{2, \dots, n + 1\},\\[2mm]
      0, &\text{if } j = 1, k \in \{2, \dots, n + 1\}, \text{ and}\\[2mm]
      0, &\text{if } k = 1, j \in \{2, \dots, n + 1\}.
   \end{cases}
    \end{align*}
    Therefore,
    \begin{align}
        &J_{\mathbb{D} \times B_n(0, 1)}(0) = \frac{\operatorname{det}G_{\mathbb{D} \times B_n(0, 1)}(0)}{\kappa_{\mathbb{D} \times B_n(0, 1)}(0)} = \frac{2\pi^{n + 1} (n + 1)^n}{n!}, \\   
        &B_{\mathbb{D} \times B_n(0, 1)}(0; \xi) = \sqrt{2|\xi_1|^2 + (n + 1)|\xi'|^2},
    \end{align}
    for $\xi \in \mathbb{C}^{n + 1}$, and
    \begin{align}
        R_{\bar{h}jk\bar{\ell}}(0) &= - \frac{\partial^2 g_{j\bar{h}}}{\partial z_k \partial \bar{z}_l}(0) + \sum_{\mu, \nu} g^{\nu \bar{\mu}}(0) \frac{\partial g_{j \bar{\mu}}}{\partial z_k}(0) \frac{\partial g_{\nu \bar{h}}}{\partial \bar{z}_l}(0)\\
        &= - \frac{\partial^2 g_{j\bar{h}}}{\partial z_k \partial \bar{z}_l}(0)\\
        &= \begin{cases} 
      -4, &\text{if } h = j = k = \ell = 1,\\[2mm]
      - (n + 1)(\delta_{kl}\delta_{jh} + \delta_{kh}\delta_{jl}),&\text{if } h, j, k, l \in \{2, \dots, n + 1\}, \text{ and}\\[2mm]
      0, &\text{otherwise.}
   \end{cases}
    \end{align}
    Hence, the Ricci curvature
    \begin{align*}
        R_{\mathbb{D} \times B_n(0, 1)}(0; \xi) &= \frac{\sum_{h,j,k,l}g^{k\bar{l}}(0)R_{\bar{h}jk\bar{l}}(0)\bar{\xi}_h \xi_j}{B^2_{\mathbb{D} \times B_n(0, 1)}(0 ; \xi)}
        = -\frac{2|\xi_1|^2 + (n + 1)|\xi'|^2}{B_{\mathbb{D} \times B_n(0, 1)}^2(0; \xi)} = -1,
    \end{align*}
    for $\xi \in \mathbb{C}^{n + 1} \setminus \{0\}$,
    the Scalar curvature
    \begin{align*}
        S_{\mathbb{D} \times B_n(0, 1)}(0) = \sum_{h, j, k, \ell} g^{j\bar{h}}(0)g^{k\bar{\ell}}(0)R_{\bar{h}jk\bar{\ell}}(0) = -(n + 1), \, \text{and}
    \end{align*}
the Kobayashi--Fuks metric
\begin{align*}
    M^F_{\mathbb{D} \times B_n(0, 1)}(0; \xi) &= B_{\mathbb{D} \times B_n(0, 1)}(0; \xi)\sqrt{(n + 2) - R_{\mathbb{D} \times B_n(0, 1)}(0; \xi)}\\
    &= (n + 3)^{1/2}(2 |\xi_1|^2 + (n + 1)|\xi'|^2)^{1/2},
\end{align*}
for $\xi \in \mathbb{C}^{n + 1} \setminus \{0\}$.
\end{proof}
We will establish localization for extremal functions (Lemma \ref{localisation lemma 4}) by employing the following theorem. Before stating the theorem, we introduce the following notation: For a plurisubharmonic function $\phi$ on $D$, let
\begin{align*}
    L^2_{(0, 1)}(D, \phi) := \left\{g = \sum_{j = 1}^{n}g_j \mathrm{d}\bar{z}_j : \int_{D} |g|^2\operatorname{e}^{-\phi} \, \mathrm{d}V = \int_{D} \left(\sum_{j = 1}^n |g_j|^2\right) \operatorname{e}^{-\phi} \, \mathrm{d}V < \infty \right\}.
\end{align*}
\begin{thm}(Hörmander \cite{Hörmander})\label{Hormander} Let $D$ be a pseudoconvex open set in $\mathbb{C}^n$ and $\phi$ any plurisubharmonic function in $D$. For every $g \in L^2_{(0, 1)}(D, \phi)$ with $\overline{\partial}g = 0$, there is a solution $u \in L^2_{\operatorname{loc}}(D)$ of the equation $\overline{\partial}u = g$ such that
\begin{align*}
    \int_{D} \frac{|u|^2e^{-\phi}}{\left(1 + |z|^2\right)^{2}}\, \mathrm{d}V \leq \int_D|g|^2e^{-\phi} \, \mathrm{d}V.
\end{align*}
\end{thm}
\subsection{Kobayashi metric}
Let $D \subset \mathbb{C}^n$ be a domain. The Kobayashi metric on $D$ is the function $M_{D}^K : D \times \mathbb{C}^{n} \to \mathbb{R}^{+}$ defined by
\begin{align}
    M_D^K(z;\xi) = \operatorname{inf}\{ \alpha > 0 : \exists f \in \mathcal{O}(\mathbb{D}, D) \text{ with } f(0) = z, \, f'(0) = \xi/ \alpha\}.
\end{align}
Holomorphic mappings are contractions with respect to the Kobayashi metric in the following sense:
if $g : D_1 \to D_2$ is a holomorphic map, then
\begin{align}
    M_{D_2}^K(g(z); g'(z)\xi) \leq M_{D_1}^K(z; \xi), \quad z \in D_1, \, \xi \in \mathbb{C}^n.
\end{align}
We will need the following formula for the Kobayashi metric to prove our result.
\begin{lem}\label{Kobayashi formula}
    The Kobayashi metric on $\mathbb{D} \times B_n(0, 1)$ at $0$ is as follow.
    \begin{align}\label{4}
        M_{\mathbb{D} \times B_n(0, 1)}^K(0; \xi) = \operatorname{max}\{|\xi_1|, |\xi'|\},
    \end{align}
    where $\xi = (\xi_1, \xi') \in \mathbb{C} \times \mathbb{C}^{n}$.
\end{lem}
\begin{proof}
    Follows from \cite{Jarnicki}*{Example $3.5.6$ and Proposition $3.7.1 (b)$}.
\end{proof}
We now introduce the Möbius and Carathéodory pseudodistance.
\begin{defn}
    Let $D$ be a domain in $\mathbb{C}^{n}$. The Möbius pseudodistance for $D$ is given by
      \begin{align}
      c_D^*(z,w) := \operatorname{sup}\left\{\left|\frac{f(z) - f(w)}{1 - f(z)\overline{f(w)}}\right|: f \in \mathcal{O}(D, \mathbb{D})\right\},
      \end{align}
      and the Carathéodory pseudodistance for $D$ is given by
      \begin{align}
      c_D(z, w) := \operatorname{tanh}^{-1}(c_D^{*}(z, w)). 
     \end{align}
where $z, w \in D$.
\end{defn}
The following theorem due to Royden \cite{Royden 1970}*{Lemma $2$}, whose proof can be found
in \cite{Graham 1975}*{Lemma $4$}, is helpful in proving localization for the Kobayashi metric (Lemma \ref{localization of Kobayashi} and \eqref{loc of kobayashi}).
\begin{thm}[\cite{Graham 1975}*{Lemma $4$}, \cite{Jarnicki}*{Proposition 7.2.9}]\label{Graham theorem}
    Let $D \subset \mathbb{C}^{n}$ be a bounded domain in and $U \subset D$ be any subdomain. Then, for $z \in U \, \text{and } \xi \in \mathbb{C}^{n}$, we have 
    \begin{align*}
        M_U^K(z; \xi) \leq C_{U}(z) M_{D}^K(z ; \xi), \text{ and}\\
        M_U^K(z; \xi) \leq \operatorname{coth}c_{D}(z, D\setminus U) M_{D}^K(z; \xi),
    \end{align*}
    where $C_U(z) := \operatorname{sup}\left\{1/r > 1 : \text{ there exists } f \in \mathcal{O}\left(\mathbb{D}, {D}\right) \text{ such that } f(0) = z, \, f(r) \in D \setminus U\right\}$, and $c_D(z, D\setminus U) := \operatorname{inf}\left\{c_D(z,w) : w \in D \setminus U\right\}$.
\end{thm}
\section{Main Results}\label{Cone}
We recall some key facts about exponentially flat domains; for details, see \cite{Ravi}. This Section is divided into two parts. In the first part, we will prove the nontangential asymptotic behaviour of the Bergman canonical invariant, Ricci and Scalar curvatures of the Bergman metric, and Kobyashi--Fuks metric at exponentially flat infinite type boundary points of bounded smooth pseudoconvex domains in $\mathbb{C}^{n + 1}$. In the second part, we will prove the nontangential asymptotic behaviour of the Kobayashi metric at exponentially flat infinite type boundary points of bounded smooth domains in $\mathbb{C}^{n + 1}$.

Let $D \subset \mathbb{C}^{n + 1}$ be a bounded smooth domain with $0 \in bD$ and $q(t)$ be an $(\alpha, N)$ cone type curve approaching $0$. If $0$ is an exponentially flat boundary point, then there exists a neighbourhood $U$ of origin, such that
\begin{align}\label{*}
    D \cap U = \{z \in U: \rho(z) < 0\},
\end{align}
where the defining function $\rho$ is as in \eqref{1}. 

Define $T^1_t: \mathbb{C}^{n+1} \to \mathbb{C}^{n+1}$ by
\begin{align}
    T^1_t(w) = (w_1 - i\operatorname{Im}q_1(t), w_2, \dots, w_{n+1}),
\end{align}
and introduce the Hermitian map $R_t^1: \mathbb{C}^{n + 1}  \to \mathbb{C}^{n + 1}$ given by
\begin{align}
R_t^1(x) = \begin{cases}
x, & \text{if } r(t) := \| (0, q_2(t), \dots, q_{n+1}(t)) \| = 0,\\[2mm]
e_1, & \text{if } x = e_1 \text{ and } r(t) \neq 0, \text{ and}\\[2mm]
e_2, & \text{if } x = \frac{(0, q_2(t), \dots, q_{n+1}(t))}{\| (0, q_2(t), \dots, q_{n+1}(t)) \|} \text{ and } r(t) \neq 0.
\end{cases}
\end{align}
Consequently, $\Tilde{q}(t) := R^1_t\circ T^1_t(q(t)) = (\operatorname{Re}q_1(t), r(t), 0,\dots, 0) $, and $\operatorname{Re}q_1(t) < 0$.

Choose $\delta_0 > 0$ such that 
\begin{align}
    (-4\delta_0, 4\delta_0)^2 \times B_n(0, 4\delta_0) \subset U.
\end{align} 
Since $q(t) \to 0$ as $t \to 0^{+}$, $q(t) \in (-\delta_0, \delta_0)^2 \times B_n(0, \delta_0) \subset U$ for all sufficiently small $t > 0$.  

Thus, $\Tilde{q}(t) \in \Omega \cap R^1_t\left(T^1_t \left((-\delta_0, \delta_0)^2 \times B_n(0, \delta_0)\right)\right) \cap 
C_{\alpha, N} \subset \Omega \cap U \cap C_{\alpha,N}$ for all sufficiently small $t > 0$, where $\Omega = \big\{z \in \mathbb{C}^{n+1} : \rho(z) < 0\big\}$.

For small $t > 0$, there exists a unique point $p(t) \in b\Omega$ such that 
\begin{align*}
   d(t) := \operatorname{dist}(\Tilde{q}(t), b\Omega) = |\Tilde{q}(t) - p(t)|, \, \, \, p_1(t) \in \mathbb{R}, \, \, \, p_2(t) \geq 0, \, \, \, \text{and} \, \, \, p_j(t) = 0 \text{ for } j = 3, \dots, n+1. 
\end{align*}
As $p(t) \to 0$ when $t \to 0^{+}$, it follows that $p(t) \in B_{n+1}(0, \delta_0)$ for all sufficiently small $t > 0$.

Since $\rho(z) = \operatorname{Re}z_1 + \phi\left(|z_2|^2 + \dots + |z_{n + 1}|^2\right)$, the gradient $\nabla\rho\left(p(t)\right) = \left(1, 2p_2(t)\phi'(p_2(t)^2), 0\right) $.

Now, define $T^2_t : \mathbb{C}^{n+1} \to \mathbb{C}^{n+1}$ by
\begin{align}
    T^2_t(w) = w - p(t),
\end{align}
and introduce a Hermitian linear map $R^2_t : \mathbb{C}^{n+1} \to \mathbb{C}^{n+1}$ which maps 
\begin{align}
    \frac{\nabla \rho (p(t))}{\norm{\nabla \rho(p(t))}} &\mapsto e_1, \\
    \frac{(-2p_2(t)\phi'(p_2(t)^2), 1, 0)}{\norm{\nabla \rho(p(t))}} &\mapsto e_2, \\
    e_j &\mapsto e_j, \quad \text{for } j = 3, \dots, n+1.  
\end{align}
Define $\gamma_t = R^2_t \circ T^2_t $ and $A(t) = \|\nabla \rho(p(t))\|$, therefore
\begin{align}
\left(\gamma_t\right)^{-1}(z_1, \dots, z_{n+1}) &= \bigg(\frac{z_1 - 2p_2(t)z_2\phi'(p_2(t)^2)}{A(t)} - \phi(p_2(t)^2), \frac{z_2 + 2z_1p_2(t)\phi'(p_2(t)^2)}{A(t)} \nonumber\\
& \hspace{80mm} + p_2(t), z_3, \dots, z_{n+1}\bigg), \\
\rho \circ \gamma_t^{-1}(z_1, \dots, z_{n+1}) &= \frac{\operatorname{Re}z_1}{A(t)} - \frac{2p_2(t)\phi'(p_2(t)^2)}{A(t)}\operatorname{Re}z_2 - \phi(p_2(t)^2) \nonumber\\
& \hspace{2mm}  + \phi\left(\left|\frac{z_2}{A(t)} + p_2(t) + \frac{2p_2(t)\phi'(p_2(t)^2)}{A(t)}z_1\right|^2 + |z_3|^2 + \dots + |z_{n+1}|^2\right).
\end{align}
Since $W := (-\delta_0/10, \delta_0/10)^2 \times B_n(0, \delta_0/10) \subset \gamma_t \circ  R^1_t \circ T^1_t(U)$ for all sufficiently small $t > 0$, $\gamma_t\left(R^1_t \circ T^1_t(D)\right) \cap W = \gamma_t(\Omega) \cap W$. Let $\epsilon > 0$.
Define
\begin{align}
    D_t^{\epsilon} &= \left\{(z_1, \dots, z_{n+1}) \in W : \rho \circ \gamma_t^{-1}(z_1, z_2, \dots, z_{n+1}) < 0, \operatorname{Re}z_1 > - d(t)^{{1}/{(1+\epsilon)^2}}\right\}, \text{ and} \\
    \widetilde{D}_t^{\epsilon} &= \left\{z \in D_t^{\epsilon}: \left|h_{\epsilon}(z)\right| > \operatorname{exp}\left(- d(t)^{{1}/{(1+\epsilon)^2}}c_0^{\epsilon}\right)\right\},
\end{align}
where $c_0^{\epsilon} = \operatorname{cos}\left(\frac{\pi}{2(1 + \epsilon)}\right)$
and
$h_\epsilon : D_t^{\epsilon} \to \mathbb{D}$ is defined by 
\begin{align}
    h_\epsilon(z) = \operatorname{exp}\left(-(-z_1)^{{1}/{(1+ \epsilon)}}\right). 
\end{align}
For $\epsilon, t > 0$, let
\begin{align}
    d^*(t) &:= 
    \operatorname{min}\{s \in\mathbb{R}^+: se_2 + (-d(t), 0, \dots, 0) \in b\gamma_t(\Omega)\}, \\
    d_1^\epsilon(t) &:=
     \operatorname{sup} \left\{|z'| : z \in W, \phi\left(\left|\frac{z_2}{A(t)} + p_2(t) + \frac{2p_2(t)\phi'(p_2(t)^2)}{A(t)}z_1\right|^2 + \dots + |z_{n+1}|^2\right) \right.\nonumber\\
     &\left. \hspace{55mm} \leq \frac{d(t)^{\frac{1}{(1+\epsilon)}}}{A(t)} 
    + \phi(p_2(t)^2) + \frac{2p_2(t)\phi'(p_2(t)^2)}{A(t)}\operatorname{Re}(z_2)\right\},\\
    d_2^\epsilon(t) &:=
    \operatorname{sup}\left\{|z'| : z \in W, \phi\left(\left|\frac{z_2}{A(t)} + p_2(t) + \frac{2p_2(t)\phi'(p_2(t)^2)}{A(t)}z_1\right|^2 + \dots + |z_{n+1}|^2\right) \right.\nonumber\\
    &\left. \hspace{55mm} \leq \frac{d(t)^{\frac{1}{(1+\epsilon)^2}}}{A(t)} 
    + \phi(p_2(t)^2) + \frac{2p_2(t)\phi'(p_2(t)^2)}{A(t)}\operatorname{Re}(z_2)\right\},
    \end{align}
and $\Sigma : \mathbb{C}^{n+1} \to \mathbb{C}^{n+1}$
by
\begin{align}
    \Sigma(z_1, z') = \left(\frac{z_1}{d(t)}, \frac{z'}{d^*(t)}\right),
\end{align}
where $z' = (z_2, \dots, z_{n+1})$.

We have the following limiting behaviour of $\phi, \, d(t), \, d^*(t), d_1^\epsilon(t)$ and $d_2^\epsilon(t)$ at zero, which is useful for proving our results.
\begin{lem}[\cite{Ravi}*{Lemma $3.1$}]\label{3.1}
Let $D \subset \mathbb{C}^{n + 1}$ be as in \eqref{*}. Then
\begin{align*}
    &\lim_{t \to 0^+} \frac{\phi \left(p_2(t)^2\right)}{d(t)} 
    = 0, &&\lim_{t \to 0^+} \frac{\phi'\left(p_2(t)^2\right)}{d(t)} 
    = 0, && \lim_{t \to 0^{+}} \frac{p_2(t)}{d^{*}(t)}  = 0,\\ &\liminf_{t \to 0^+} \frac{d^{*}(t)}{d_1^{\epsilon}(t)} \geq \frac{1}{(1 + \epsilon)^{\frac{1}{2m}}}, \text{ and}  &&\liminf_{{t \to 0^*}}\frac{d^*(t)}{d_2^\epsilon(t)} \geq \frac{1}{(1 + \epsilon)^{\frac{1}{m}}}.
\end{align*}
\end{lem}
We also have the following limiting behaviour of $f \circ \Sigma (D_t^{\epsilon})$.
\begin{lem}[\cite{Ravi}*{Lemma $3.3$}]\label{ScalingLemma2}
    Let $D \subset \mathbb{C}^{n + 1}$ be as in \eqref{*}. Then, for every $\epsilon, \delta > 0,$ there exists $t_0(\delta, \epsilon) > 0$ such that
    \begin{align}\label{inclusions 2}
        (1-\delta)\mathbb{D}\times B_n(0,1) \subset f \circ \Sigma \left(D_t^\epsilon\right) \subset
        \mathbb{D} \times B_n(0, {d}_2^\epsilon(t)/d^*(t)),
    \end{align}
for each $ 0 < t < t_0(\delta, \epsilon
)$, where $f(z_1, z') = ((1+z_1)/(1-z_1), z')$.
\end{lem}
\subsection{Asymptotic behaviour of the Bergman invariant, Ricci curvature, Scalar curvature, and
Kobayashi--Fuks metric}
In the following lemma, we will establish localization for extremal
functions. The method of the proof is same as in \cite{Ravi}*{Lemma $3.2$}.
\begin{lem}\label{localisation lemma 4}
     Let $D \subset \mathbb{C}^{n+1}$ be as in \eqref{*}. Further, assume that $D$ is pseudoconvex. For all sufficiently small $t > 0$, let $D_t = \gamma_t\left(R_t^1 \circ T_t^1(D)\right)$. Then, for $\epsilon > 0$ and $\xi(t) \in \mathbb{C}^{n + 1} \setminus \{0\}$, we have
    \begin{align*}
            &\lim_{t \to 0^{+}}\frac{\lambda_{D_t}(-d(t), 0)}{\lambda_{D_t^{\epsilon}}(-d(t), 0)} = 1, \quad \quad \quad \quad \quad \quad
            \lim_{t \to 0^{+}}\frac{I_{D_t}((-d(t), 0); \xi(t))}{I_{D_t^{\epsilon}}((-d(t), 0); \xi(t))} = 1,\\
        &\lim_{t \to 0^{+}} \frac{M_{D_t}((-d(t), 0); \xi(t))}{M_{D_t^{\epsilon}}((-d(t), 0); \xi(t))} = 1, \text{ and} \quad
        \lim_{t \to 0^{+}} \frac{N_{D_t}(-d(t), 0)}{N_{D_t^{\epsilon}}(-d(t), 0)} = 1.
    \end{align*}
\end{lem}
\begin{proof}
Let $\epsilon > 0$. Consider a cut-off function $X \in C_c^\infty(B_1(0,1) \times B_{n}(0, 2))$ satisfying the conditions
\begin{align*}
    X = 1 \quad \text{on} \quad B_1(0, 1/2) \times B_{n}(0,1), \quad \text{and} \quad 0 \leq X \leq 1. 
\end{align*}
Define
\begin{align*}
    X_t(z) = X\left(\frac{z_1}{d(t)^{\frac{1}{(1+\epsilon)^2}}}, \frac{z'}{d_1^{\epsilon}(t)}\right) \text{ for } z \in (z_1, z') \in \mathbb{C}^{n + 1}.
\end{align*}
Let $ z \in \widetilde{D}_t^\epsilon$.
Then, $|z_1| \leq d(t)^{\frac{1}{(1+\epsilon)}}$ and
there exist $t_0(\epsilon) > 0$ such that
\begin{align*}
    \frac{|z_1|}{d(t)^{\frac{1}{(1+\epsilon)^2}}} \leq d(t)^\frac{\epsilon}{(1+\epsilon)^2} < 1/2, \text{  and  }
    \frac{|z'|}{d_1^\epsilon(t)} < 1,
\end{align*}
for each $0<t<t_0(\epsilon)$. Therefore, $X_t = 1$ on $\widetilde{D}_t^\epsilon$. Define
\begin{align*}
{A}_t = W \cap \left\{z \in \mathbb{C}^{n+1} : |\operatorname{Re}z_1| < d(t)^\frac{1}{(1+\epsilon)^2} \right\},
\end{align*}
where $W$ is defined at the beginning of the section \ref{Cone}.
Note that $X_t \in C_c^\infty(A_t)$, and $A_t \cap D_t = D_t^\epsilon$. Let $f \in A^2(D_t^\epsilon)$ and $k \in \mathbb{N}$. Set
\begin{align*}
    \alpha &= \overline{\partial}\left(fX_t h_{\epsilon}^k\right) \text{ on } D_t, \text{ and}\\
    \phi(z) &= (2n+8)\operatorname{log}|z - (-d(t),0)|, \text{ for } z \in \mathbb{C}^{n + 1}.   
\end{align*}
By using Theorem \ref{Hormander}, we have $u \in L^2_{\operatorname{loc}}(D_t)$ satisfying
\begin{align*}
    \overline{\partial}u = \alpha \text{ on } D_t,
\end{align*}and
\begin{align}\label{Hor inequality 2}
\int_{D_t} \frac{|u|^2e^{-\phi}}{(1+ |z|^2)^2} \, \mathrm{d}V \leq \int_{D_t}|\alpha|^2 e^{-\phi} \, \mathrm{d}V.
\end{align}
Now,
\begin{align}\label{estimate alpha 2}
    \int_{D_t}|\alpha|^2 e^{-\phi} \, \mathrm{d}V &= \int_{D_t\cap A_t}\frac{|f|^2|\overline{\partial}X_t|^2|h_{\epsilon}|^{2k}}{|z - (-d(t), 0)|^{(2n+8)}} \, \mathrm{d}V \nonumber\\
    &\leq \frac{Ca_t^{2k}}{d(t)^2}\int_{D_t^\epsilon \setminus 
    \widetilde{D}_t^\epsilon} \frac{|f|^2}{|z - (-d(t), 
 0)|^{(2n + 8)}}\, \mathrm{d}V \nonumber\\
    &\leq \frac{C a_t^{2k}}{d(t)^{(2n + 10)}} \int_{D_t^\epsilon \setminus \widetilde{D}_t^\epsilon}|f|^2 \, \mathrm{d}V \nonumber\\
    & \leq \frac{Ca_t^{2k}}{d(t)^{(2n+10)}} \norm{f}_{L^2(D_t^\epsilon)}^2,
\end{align}
for each $0 < t < t_0(\epsilon)$, where $t_0(\epsilon)$ is small positive number, $a_t = \operatorname{exp}\left(-c_0^{\epsilon}d(t)^{{1}/{(1+\epsilon)^2}}\right)$, and $C > 0$, which may change at each step, depends on both the first order derivatives of $X$ and $\epsilon > 0$.
Now we estimate the left side of \eqref{Hor inequality 2}.
\begin{align}\label{estimate u_0 2}
   \int_{D_t} \frac{|u|^2e^{-\phi}}{(1+ |z|^2)^2} \, \mathrm{d}V =\int_{D_t} \frac{|u|^2}{(1+|z|^2)^2 \, |z - (-d(t),0)|^{(2n+8)}} \, \mathrm{d}V \geq C_0 \norm{u}_{L^2(D_t)}^2,
\end{align}
where $C_0 > 0$
depends only on the domain $D$.
From \eqref{Hor inequality 2},
\eqref{estimate alpha 2}, and \eqref{estimate u_0 2}, we have
\begin{align}\label{res align}
    \norm{u}_{L^2(D_t)}^2 \leq \frac{C a_t^{2k}}{C_0d(t)^{(2n + 10)}}\norm{f}_{L^2(D_t^\epsilon)}^2, 
\end{align}
and 
\begin{align}\label{19}
    \int_{D_t} \frac{|u|^2}{(1+|z|^2)^2 \, |z - (-d(t),0)|^{(2n+8)}} \, \mathrm{d}V \leq \frac{C a_t^{2k}}{d(t)^{(2n + 10)}}\norm{f}_{L^2(D_t^\epsilon)}^2.
\end{align}
From the above equation \eqref{19}, we get
\begin{align}
        \frac{\partial^{|\alpha| + |\beta|} u}{\partial z^{\alpha} \partial \bar{z}^{\beta}}(-d(t), 0) = 0 \text{ for all multi-indices $\alpha, \beta$ with $|\alpha| + |\beta| \leq 2$.}
\end{align}
Set
\begin{align}\label{641}
    g_k = {X_t f h_{\epsilon}^k - u} \text{ on $D_t$}.
\end{align}
Then $g_k \in A^2(D_t)$. Moreover,
\begin{align*}
        \norm{g_k}_{L^2(D_t)} &\leq \norm{X_t f h_{\epsilon}^k}_{L^2(D_t)} + \norm{u}_{L^2(D_t)}\\ 
        &\leq \norm{f }_{L^2(D_t^{\epsilon})} + \norm{u}_{L^2(D_t)}\\
        &\leq \left( 1+ \sqrt{\frac{C}{C_0}} \frac{a_t^k}{d(t)^{n + 5}}\right)\norm{f}_{L^2(D_t^{\epsilon})}.
\end{align*}
Let $f \in A^2(D_t^{\epsilon})$ be the maximum function for $\lambda_{D_t^{\epsilon}}^{\ell}(-d(t), 0)$, where $\ell \in \{1, \dots, n + 1\}$, i.e.,
\begin{align*}
    \lambda_{D_t^{\epsilon}}^{\ell}(-d(t), 0) = \left|\frac{\partial f}{\partial z_{\ell}} (-d(t), 0)\right|^2,
\end{align*}
$f(-d(t), 0) = 0, \, \frac{\partial f}{\partial z_j}(-d(t), 0) = 0$, for $ 1 \leq j < \ell$, and $\norm{f}_{L^2(D_t^{\epsilon})} = 1$.

For any $z \in D_t$, set $\mathcal{G}_k(z) = g_k(z)/\norm{g_k}_{L^2(D_t)} \in A^2(D_t)$. Then,
\begin{align*}
    \mathcal{G}_{k}(-d(t), 0) = 0, \, \frac{\partial \mathcal{G}_k}{\partial z_j}(-d(t), 0) = 0, \text{ for } 1 \leq j < \ell, \text{ and } \norm{\mathcal{G}_k}_{L^2(D_t)} = 1. 
\end{align*}
Therefore, we have
\begin{align*}
    \lambda_{D_t}^{\ell}(-d(t), 0) &\geq \left|\frac{\partial \mathcal{G}_k}{\partial z_{\ell}}(-d(t), 0)\right|^2\\
    &= \norm{g_k}_{L^2(D_t)}^{-2} |h_{\epsilon}^k(-d(t), 0)|^{2} \lambda_{D_t^{\epsilon}}^{\ell}(-d(t), 0)
\end{align*}
Hence
\begin{align*}
    \frac{\lambda_{D_t}^{\ell}(-d(t), 0)}{\lambda_{D_t^{\epsilon}}^{\ell}(-d(t), 0)} \geq \frac{|h^k_{\epsilon}(-d(t), 0)|^2}{\left( 1+ \sqrt{\frac{C}{C_0}} \frac{a_t^k}{d(t)^{n + 5}}\right)^2},
\end{align*}
for each $0 < t < t_0(\epsilon)$. For every $\epsilon > 0$, there exist $c_{\epsilon} > 0$, such that
\begin{align}
    \frac{1}{(1 + \epsilon)^2} < c_{\epsilon} < \frac{1}{(1+\epsilon)}.
\end{align}
Choosing $k$ is the greatest integer of $d(t)^{-c_{\epsilon}}$, we get
\begin{align}\label{671}
    \liminf_{t \to 0^{+}}\frac{\lambda_{D_t}^{\ell}(-d(t), 0)}{\lambda_{D_t^{\epsilon}}^{\ell}(-d(t), 0)} \geq 1.
\end{align}
By using \eqref{671} and the fact that $\lambda^{\ell}$ is monotonically decreasing with respect to the domain, we get
\begin{align}
    \lim_{t \to 0^{+}} \frac{\lambda_{D_t}^{\ell}(-d(t), 0)}{\lambda_{D_t^{\epsilon}}^{\ell}(-d(t), 0)} = 1,
\end{align}
for each $\ell \in \{1, \dots, n + 1\}$.
By using \cite{Ravi}*{Lemma $3.2$ and Proposition $2.1$}, we get the localization of the Bergman kernel, i.e.,
\begin{align}\label{loc of Bergman kernel}
    \lim_{t \to 0^{+}} \frac{\kappa_{D_t}(-d(t), 0)}{\kappa_{D_t^{\epsilon}}(-d(t), 0)} = 1.
\end{align} 
Therefore, $J$ also satisfies the localization.

Next, let $f \in A^2(D_t^{\epsilon})$ be the maximum function for $I_{D_t^{\epsilon}}((-d(t), 0); \xi(t))$, i.e., 
\begin{align*}
   I_{D_t^{\epsilon}}((-d(t), 0); \xi(t)) = \xi(t)f''(-d(t), 0)\overline{G}^{-1}_{D
_t^{\epsilon}}(-d(t), 0) \overline{f''(-d(t), 0)}\xi(t)^{*},
\end{align*}
$f(-d(t), 0) = 0, \, f'(-d(t), 0) = 0$, and $\norm{f}_{L^2(D_t^{\epsilon})} = 1$.

For any $z \in D_t$, set $G_k(z) = g_k(z) / \norm{g_k}_{L^2(D_t)} \in A^2(D_t)$, where $g_k$ is as in \eqref{641}. Then, 
\begin{align*}
    G_k(-d(t), 0) = G_K'(-d(t), 0) = 0 \text{ and } \norm{G_k}_{L^2(D_t)} = 1.
\end{align*}
Therefore, by using Proposition \ref{Krantz-Yu} in the third step below, we have
\begin{align*}
        I_{D_t}((-d(t), 0); \xi(t)) &\geq \xi(t)G_k''(-d(t), 0)\overline{G}^{-1}_{D_t}(-d(t), 0)\overline{G_k''(-d(t), 0)}\xi(t)^{*}\\
        &= \norm{g_k}_{L^2(D_t)}^{-2}|h_{\epsilon}^k(-d(t), 0)|^2 \xi(t)f''(-d(t), 0)\overline{G}_{D_t}^{-1}(-d(t), 0)\overline{f''(-d(t), 0)}\xi(t)^*\\
        &\geq \norm{g_k}_{L^2(D_t)}^{-2}|h_{\epsilon}^k(-d(t), 0)|^2 \frac{\kappa_{D_t}(-d(t), 0)}{\kappa_{D_t^{\epsilon}}(-d(t), 0)}I_{D^{\epsilon}_t}((-d(t), 0); \xi(t)).
\end{align*}
Hence
\begin{align*}
     \frac{ I_{D_t}((-d(t), 0); \xi(t))}{I_{D^{\epsilon}_t}((-d(t), 0); \xi(t))} \geq \frac{|h^k_{\epsilon}(-d(t), 0)|^2}{\left( 1+ \sqrt{\frac{M}{M_0}} \frac{a_t^k}{d(t)^{n + 5}}\right)^2} \frac{\kappa_{D_t}(-d(t), 0)}{\kappa_{D_t^{\epsilon}}(-d(t), 0)},
\end{align*}
for each $0 < t < t_0(\epsilon)$.
For every $\epsilon > 0$, there exist $c_{\epsilon} > 0$, such that
\begin{align}
    \frac{1}{(1 + \epsilon)^2} < c_{\epsilon} < \frac{1}{(1+\epsilon)}.
\end{align}
Choosing $k$ is the greatest integer of $d(t)^{-c_{\epsilon}}$ and using \eqref{loc of Bergman kernel}, we get
\begin{align}\label{96}
\liminf_{t \to 0^+} {\frac{I_{D_t}(\left(-d(t), 0\right); \xi(t))}{I_{D_t^\epsilon}\left((-d(t), 0); \xi(t)\right)}} \geq 1,
\end{align}for each $\epsilon > 0$.
On the other hand,
\begin{align}\label{M}
    I_{D}(z; u) = \frac{M_{D}(z; u)}{\kappa_{D}^n(z) J_{D}(z)}, \quad \forall z \in D, u \in \mathbb{C}^{n}.
\end{align}   
Since $M_{D}$ is monotonically decreasing with respect to domain, 
\begin{align*}
    \frac{I_{D_t}(\left(-d(t), 0\right); \xi(t))}{I_{D_t^\epsilon}\left((-d(t), 0); \xi(t)\right)} \leq \frac{\kappa^n_{D_t^{\epsilon}}(-d(t), 0)}{\kappa^n_{D_t}(-d(t), 0)} \frac{J_{D_t^{\epsilon}}(-d(t), 0)}{J_{D_t}(-d(t), 0)}.
\end{align*}
By using localization of $J$ and the Bergman kernel \eqref{loc of Bergman kernel}, we get
\begin{align}\label{98}
    \limsup_{t \to 0^+} {\frac{I_{D_t}(\left(-d(t), 0\right); \xi(t))}{I_{D_t^\epsilon}\left((-d(t), 0); \xi(t)\right)}}
    \leq 1,
\end{align}
for each $\epsilon > 0$. We now use \eqref{96} and \eqref{98}, we get
\begin{align*}
    \lim_{t \to 0^+} {\frac{I_{D_t}(\left(-d(t), 0\right); \xi(t))}{I_{D_t^\epsilon}\left((-d(t), 0); \xi(t)\right)}} = 1,
\end{align*}
for each $\epsilon > 0$.

The localization of $M$ follows from
\eqref{M} and the localization of $\kappa, \, J, \, \text{and } I$.

Similarly, one can prove the localization of $N$. 
\end{proof}
In the following lemma, we compare the extremal functions of $D_t$ to those of $\mathbb{D} \times B_n(0, 1)$.
\begin{lem}\label{3.6}
     Let $D \subset \mathbb{C}^{n+1}$ be as in \eqref{*}. Suppose $D$ is pseudoconvex and $q$ is an $(\alpha, N)$-cone type curve approaching $0$. Then, for $\xi(t) \in \mathbb{C}^{n + 1}\setminus \{0\}$,
     \begin{align*}
          \lim_{t \to 0^+} \frac{(2d(t))^{2(n + 2)}(d^{*}(t))^{2n(n + 2)}\lambda_{D_t}(-d(t), 0)}{\lambda_{\mathbb{D} \times B_n(0, 1)}(0)} &= 1,\\
          \lim_{t \to 0^+} \frac{(2d(t))^{2(n + 2)}(d^{*}(t))^{2n(n + 2)}M_{D_t}((-d(t), 0); \xi(t))}{ M_{\mathbb{D} \times B_n(0, 1)}(0; (f \circ \Sigma)'(-d(t), 0) \xi(t))} &= 1, \, \text{and}\\
          \lim_{t \to 0^+} \frac{(2d(t))^{2(2n + 3)}(d^{*}(t))^{2n(2n + 3)} N_{D_t}(-d(t), 0)}{N_{\mathbb{D} \times B_n(0, 1)}(0)} &= 1.
     \end{align*}
\end{lem}
\begin{proof}
    Since
    \begin{align}
        M_{D_t^{\epsilon}}((-d(t), 0); \xi(t)) = {(2d(t))^{-2(n + 2)}d^*(t)^{-2n(n + 2)} M_{f \circ \Sigma (D_t^{\epsilon})}(0; (f \circ \Sigma)'(-d(t), 0)\xi(t))}.
    \end{align}
    By using Lemma \ref{ScalingLemma2} and monotonicity of $M$ with respect to domain (see Proposition \ref{Krantz-Yu}), we have
    \begin{align}
        M_{f\circ \Sigma(D_t^{\epsilon})}(0; (f \circ \Sigma)'(-d(t),0)\xi(t)) &\leq M_{(1-\delta)\mathbb{D} \times B_n(0,1)}(0; (f \circ \Sigma)'(-d(t),0)\xi(t))\nonumber\\
        &= \frac{M_{\mathbb{D} \times B_n(0,1)}(0; (f \circ \Sigma)'(-d(t),0)\xi(t))}{(1 - \delta)^{2\left((n + 1)(n + 2) + 1 \right)}},
    \end{align}
    which implies
    \begin{align}\label{109}
        \frac{M_{f\circ \Sigma(D_t^{\epsilon})}(0; (f \circ \Sigma)'(-d(t),0)\xi(t))}{M_{\mathbb{D} \times B_n(0,1)}(0; (f \circ \Sigma)'(-d(t),0)\xi(t))} \leq (1 - \delta)^{-2(n^2 + 3n + 3)},
    \end{align}
    for each $0 < t < t_0(\delta, \epsilon)$.

    Again using Lemma \ref{ScalingLemma2}, the monotonicity of $M$ with respect to domain, and the inequality $d^{*}(t) \leq d_2^{\epsilon}(t)$ (follows from the definition), we have
    \begin{align}
        M_{(f\circ \Sigma(D_t^{\epsilon}))}(0, ((f \circ \Sigma)'(-d(t),0)\xi(t))) &\geq M_{(\mathbb{D} \times B_n(0, d_2^{\epsilon}(t)/d^{*}(t)))}{(0, ((f \circ \Sigma)'(-d(t),0))\xi(t)))}\nonumber\\
        &\geq M_{((d_2^{\epsilon}(t)/d^{*}(t))\mathbb{D} \times B_n(0, 1))}{(0, ((f \circ \Sigma)'(-d(t),0))\xi(t)))}\nonumber\\
        &= \frac{M_{(\mathbb{D} \times B_n(0,1))}{(0, ((f \circ \Sigma)'(-d(t),0))\xi(t)))}}{\left(\frac{d_2^{\epsilon}(t)}{d^{*}(t)}\right)^{2(n^2 + 3n + 3)}}\nonumber,
    \end{align}
    which implies
    \begin{align}\label{110}
        \frac{M_{f\circ \Sigma(D_t^{\epsilon})}(0; (f \circ \Sigma)'(-d(t),0)\xi(t))}{M_{\mathbb{D} \times B_n(0,1)}{(0; (f \circ \Sigma)'(-d(t),0)\xi(t))}} \geq \left(\frac{d^*(t)}{d_2^{\epsilon}(t)}\right)^{2(n^2 + 3n + 3)},
    \end{align}
    for each $0 < t < t_0(\delta, \epsilon)$.
    Now,
    \begin{align}\label{111}
        &\frac{(2d(t))^{2(n + 2)} d^*(t)^{2n(n + 2)}{M_{D_t}((-d(t),0); \xi(t))}}{M_{\mathbb{D} \times B_n(0,1)}(0; (f \circ \Sigma)'(-d(t),0)\xi(t))}  \nonumber\\
        &\,\,\,\,\,\,\,\,\,=\frac{{{M_{f\circ \Sigma(D_t^{\epsilon})}(0; (f \circ \Sigma)'(-d(t),0)\xi(t))}}}{M_{\mathbb{D} \times B_n(0,1)}(0; (f \circ \Sigma)'(-d(t),0)\xi(t))} \cdot \frac{(2d(t))^{2(n + 2)}d^*(t)^{2n(n + 2)}M_{D_t}((-d(t),0); \xi(t))}{{{M_{f\circ \Sigma(D_t^{\epsilon})}(0; (f \circ \Sigma)'(-d(t),0)\xi(t))}}}\nonumber\\
        &\,\,\,\,\,\,\,\,\,=\frac{{{M_{f\circ \Sigma(D_t^{\epsilon})}(0; (f \circ \Sigma)'(-d(t),0)\xi(t))}}}{M_{\mathbb{D} \times B_n(0,1)}(0; (f \circ \Sigma)'(-d(t),0)\xi(t))} \cdot \frac{M_{D_t}((-d(t),0); \xi(t))}{{{M_{D_t^{\epsilon}}((-d(t), 0); \xi(t))}}}.
    \end{align}
    By Lemma \ref{localisation lemma 4}, we have
    \begin{align}\label{112}
        \lim_{t \to 0^+} \frac{M_{D_t}((-d(t),0); \xi(t))}{M_{D_t^{\epsilon}}((-d(t),0); \xi(t))} = 1,
    \end{align}
    for $\epsilon > 0$.

    From \eqref{111}, \eqref{109} and \eqref{112}, we have
    \begin{align}
        \limsup_{t \to 0^+}\frac{(2d(t))^{2(n + 1)} d^*(t)^{2n(n + 1)}{M_{D_t}\left((-d(t),0); \xi(t)\right)}}{M_{\mathbb{D} \times B_n(0,1)}\left(0; (f \circ \Sigma)'(-d(t),0)\xi(t)\right)} \leq {(1 - \delta)^{-2(n^2 + 3n + 3)}} ,
    \end{align}
    for each $\delta > 0$. Hence
    \begin{align}
        \limsup_{t \to 0^+}\frac{(2d(t))^{2(n + 1)} d^*(t)^{2n(n + 1)}{M_{D_t}((-d(t),0); \xi(t))}}{M_{\mathbb{D} \times B_n(0,1)}(0; (f \circ \Sigma)'(-d(t),0)\xi(t))} \leq {1}.
    \end{align}
    Similiarly, from \eqref{111}, \eqref{110}, \eqref{112} and Lemma \ref{3.1}, we have
    \begin{align}
        \liminf_{t \to 0^{+}} \frac{(2d(t))^{2(n + 1)} d^*(t)^{2n(n + 1)}{M_{D_t}((-d(t),0); \xi(t))}}{M_{\mathbb{D} \times B_n(0,1)}(0; (f \circ \Sigma)'(-d(t),0)\xi(t))}  \geq \frac{1}{(1 + \epsilon)^{2(n^2 + 3n + 3)/m}},
    \end{align}
    for $\epsilon > 0$. Hence
    \begin{align}
        \liminf_{t \to 0^{+}} \frac{(2d(t))^{2(n + 1)} d^*(t)^{2n(n + 1)}{M_{D_t}((-d(t),0); \xi(t))}}{M_{\mathbb{D} \times B_n(0,1)}(0; (f \circ \Sigma)'(-d(t),0)\xi(t))}  \geq 1.
    \end{align}
    So, we have
    \begin{align}
         \lim_{t \to 0^+} \frac{(2d(t))^{2(n + 1)} d^*(t)^{2n(n + 1)}{M_{D_t}((-d(t),0); \xi(t))}}{M_{\mathbb{D} \times B_n(0,1)}(0; (f \circ \Sigma)'(-d(t),0)\xi(t))} = 1,
    \end{align}
    for $\xi(t) \in \mathbb{C}^{n + 1} \setminus \{0\}$.

    Similarly, we have 
     \begin{align*}
          \lim_{t \to 0^+} \frac{(2d(t))^{2(n + 2)}(d^{*}(t))^{2n(n + 2)}\lambda_{D_t}(-d(t), 0)}{\lambda_{\mathbb{D} \times B_n(0, 1)}(0)} &= 1, \text{ and}\\
          \lim_{t \to 0^+} \frac{(2d(t))^{2(2n + 3)}(d^{*}(t))^{2n(2n + 3)}N_{D_t}(-d(t), 0)}{N_{\mathbb{D} \times B_n(0, 1)}(0)} &= 1.
     \end{align*}
\end{proof}
We now use Lemma \ref{Formula for kernel} and Lemma \ref{3.6} to prove the Theorem \ref{curvatures}.
\begin{proof}[Proof of Theorem \ref{curvatures}] 
Let $\xi \in \mathbb{C}^{n + 1} \setminus \{0\}$. We denote the complex normal and complex tangential components of $\xi$ with respect to $\pi(q(t))$ as $\xi_{N,t}$ and $\xi_{T,t}$, respectively. This results in the decomposition $\xi = \xi_{N,t} + \xi_{T,t}$, where $\xi_{T,t} \in T_{\pi(q(t))}^{\mathbb{C}}(bD)$ and $\xi_{N,t} \perp T_{\pi(q(t))}^{\mathbb{C}}(bD)$.

Expressing this in detail:
\begin{align}
     \xi_{N, t} &= \langle R^2_t \circ R^1_t(\xi), e_1\rangle (R_t^2 \circ R_t^1)^{-1}(e_1), \text{ and}\\
     \xi_{T,t} &= \sum_{j = 2}^{n + 1}\langle R^2_t \circ R^1_t(\xi), e_j\rangle (R_t^2 \circ R_t^1)^{-1}(e_j).
\end{align}
Furthermore, we have $|\xi_{N, t}| = |\langle R^2_t \circ R^1_t(\xi), e_1 \rangle|$, and $|\xi_{T,t}| = \sqrt{\sum_{j = 2}^{n + 1}|\langle R^2_t \circ R^1_t(\xi), e_j\rangle|^2}$.

Then,
\begin{align*}
    \frac{J_D(q(t))}{J_{\mathbb{D} \times B_n(0, 1)}(0)} &= \frac{J_{D_t}(-d(t), 0)}{J_{\mathbb{D} \times B_n(0, 1)}(0)} \quad (\text{since } D_t = \gamma_t \circ R_t^1 \circ T_t^1(D))\\
    &= \frac{\lambda_{D_t}(-d(t), 0)}{\lambda_{\mathbb{D} \times B_n(0,1)}(0)} \cdot \frac{\kappa_{\mathbb{D} \times B_n(0, 1)}^{n + 2}(0)}{\kappa_{D_t}^{n + 2}(-d(t), 0)} \quad (\text{from Proposition \ref{Pro 2.1}})\\
    &= \frac{(2d(t))^{2(n + 2)}(d^{*}(t))^{2n(n + 2)}\lambda_{D_t}(-d(t), 0)}{\lambda_{\mathbb{D} \times B_n(0,1)}(0)} \cdot \left(\frac{d(t)^{-2} d^{*}(t)^{-2n}}{4\pi \operatorname{vol}B_n(0, 1)\kappa_{D}(q(t))}\right)^{n + 2}
\end{align*}
By using Lemma \ref{3.6} and \cite{Ravi}*{Theorem $1.4$}, we get
\begin{align}
    \lim_{t \to 0^{+}} \frac{J_D(q(t))}{J_{\mathbb{D} \times B_n(0, 1)}(0)} = 1.
\end{align}
Similarly, we have  
    \begin{align}
            &\lim_{t \to 0^{+}} \frac{(n  + 1) - R_{D}(q(t) ; \xi)}{(n + 1) - R_{\mathbb{D} \times B_n(0, 1)}(0; (f \circ \Sigma)'(-d(t), 0) (R_t^2 \circ R_t^{1}(\xi)))} = 1,\\
            &\lim_{t \to 0^{+}} \frac{n(n + 1) - S_{D}(q(t))}{n(n + 1) - S_{\mathbb{D} \times B_n(0, 1)}(0)} = 1.
    \end{align}
    We now use Lemma \ref{Formula for kernel} to get the result, i.e.,
    \begin{align*}
        &\lim_{t \to 0^{+}} {J_{D}(q(t))} = \frac{2\pi^{n + 1} (n + 1)^n}{n!} ,\\
        &\lim_{t \to 0^{+}} {R_{D}(q(t); \xi)} = -1,\\
        &\lim_{t \to 0^{+}}{S_{D}(q(t))} = -(n + 1),
    \end{align*}
    for each $\xi \in \mathbb{C}^{n + 1} \setminus \{0\}.$
\end{proof}
\begin{proof}[Proof of Theorem \ref{1.4}]
The proof follows by using Theorem \ref{curvatures} and \cite{Ravi}*{Theorem $1.4$}.
\end{proof}
\subsection{Asymptotic behaviour of the Kobayashi metric}
In the following lemma, we localize the Kobayashi metric near the origin.  
\begin{lem}\label{localization of Kobayashi}
    Let $D$ be as in \eqref{*}. Suppsose $q$ is an $(\alpha, N)$-cone type curve approaching at $0$. Then, for $\xi \in \mathbb{C}^{n + 1} \setminus \{0\}$, we have
    \begin{align}
        \lim_{t \to 0^{+}} \frac{M^K_D(q(t); \xi)}{M^K_{D \cap W_t}(q(t); \xi)} = 1,
    \end{align}
    where $W_t = (\gamma_t \circ R_t^1 \circ T_t^1)^{-1}W$.
\end{lem}
\begin{proof}
    By using Theorem \ref{Graham
    theorem}, we have
    \begin{align}
        M_{D \cap W_t}^{K}(z; \xi) \leq C_{D \cap W_t}(z) M_{D}^{K}(z; \xi), \quad z \in D \cap W_t \text{ and } \xi \in \mathbb{C}^{n + 1} \setminus \{0\},
    \end{align}
    where $C_{D \cap W_t}(z) = \operatorname{sup}\{1/r > 1 : \exists f \in \mathcal{O}(\mathbb{D}, D) \text{ such that } f(0) = z, f(r) \in D \setminus W_t\}$. For sufficiently small $t > 0$, there exists a neighbourhood $W_0 \subset W_t$ of $0$. Let $V \subset \subset W_0$ be a neighbourhood of $0$. Then, for $z \in V$, we have
    \begin{align}
        C_{D \cap W_t}(z) \leq \operatorname{sup}\left\{1/r > 1 : \exists f \in \mathcal{O}\left(\mathbb{D}, D\right) \text{ such that } f(0) = z, f(r) \in D \setminus W_0\right\} =: C(z),
    \end{align}
    for sufficiently small $t > 0$.
By using Montel's
theorem, there exists $f \in \mathcal{O}(\mathbb{D}, \overline{D})$ such that $f(0) = z, \, f(1/C(z)) \in \overline{D \setminus W_0}$. For sufficiently small $t > 0$, we have $q(t) \in V$. 

To prove the Lemma, it is enough to prove that 
\begin{align}\label{limit of C}
    \lim_{t \to 0^{+}}C(q(t)) = 1.
\end{align}
If \eqref{limit of C} is not true,
then there exists a sequence $\langle q(t_n) \rangle$ such that 
\begin{align}
    \lim_{t \to 0^{+}}C(q(t_n)) = \ell \in (1, \infty].
\end{align}
For each $C(q(t_n))$, there exists $f_n \in \mathcal{O}(\mathbb{D}, \overline{D})$ such that $f_n(0) = q(t_n), \, f_n(1/C(q(t_n))) \in \overline{D \setminus W_0}$.
By using Montel's theorem, there exists $f \in \mathcal{O}(\mathbb{D}, \overline{D})$ such that 
\begin{align}
   f(0) = 0, \, \ell < \infty \text{ and } f(1 / \ell) \in \overline{D \setminus W_0}. 
\end{align}
As $f = (f_1, \dots, f_{n + 1}) : f^{-1}(V) \to \overline{D} \cap V$ is a holomorphic function and 
\begin{align*}
    \overline{D} \cap V \subset \{z \in V: \operatorname{Re}z_1 + \phi\left(|z_2|^2 + \dots + |z_{n + 1}|^2\right) \leq 0\},
\end{align*}
$f_1$ is nonconstant holomorphic function in $f^{-1}(V)$ and $f_1(0) = 0$. Then, $\operatorname{Re}f_1(z) \leq 0$ for each $z \in f^{-1}(V)$ which contradicts the open mapping theorem. 
Therefore, \eqref{limit of C} is true.
\end{proof}
We now prove the boundary behaviour of the Kobayashi metric at exponentially flat infinite type boundary points of bounded smooth domains in $\mathbb{C}^{n + 1}$.
\begin{proof}[Proof of Theorem \ref{1.5}] We will prove the theorem in two steps. We first localize the Kobayashi metric further. We then use method of scaling to prove the theorem.

\textbf{Step $1$.}
Let $\xi \in \mathbb{C}^{n + 1} \setminus \{0\}$.
Then,
\begin{align}\label{94}
        M^K_{D \cap W_t}\left(q(t); \xi\right) = M^K_{D_t \cap W}\left((-d(t), 0);
        R^2_t \circ R^1_t(\xi)\right),
\end{align} 
where $W_t = (\gamma_t \circ R_t^1 \circ T_t^1)^{-1}(W)$ and $D_t = (\gamma_t \circ R_t^1 \circ T^1_t)(D)$.

Let $ \epsilon > 0$. Recall,
\begin{align}
{D}_t^{\epsilon} := \left\{z \in W: \rho \circ \gamma_t^{-1}(z) < 0, \operatorname{Re}z_1 > -d(t)^{\frac{1}{(1+\epsilon)^2}}\right\}.
\end{align}
By Theorem \ref{Graham theorem}, we have
\begin{align}\label{A_1}  
 M_{{D}_t^{\epsilon}}^K\left((-d(t), 0); R^2_t \circ R^1_t(\xi)\right)& \nonumber\\
    \leq \operatorname{coth}&c_{D_t \cap W}\left((-d(t), 0),D_t \cap W\setminus{D}_t^\epsilon\right)M_{D_t \cap W}^K\left((-d(t), 0); R^2_t \circ R^1_t(\xi)\right).
\end{align} 
Let $h_{\epsilon} : D_t \cap W \to \mathbb{D}$ be defined by 
\begin{align*}
    h_{\epsilon}(z) = \operatorname{exp}(-(-z_1)^{\frac{1}{1 + \epsilon}}).
\end{align*}
So,
\begin{align}\label{A_2}
    c_{D_t \cap W}^*\left((-d(t),0), z\right) &\geq \left|\frac{h_{\epsilon}(-d(t), 0) - h_\epsilon (z)}{1 - h_\epsilon (-d(t), 0))\overline{h_{\epsilon}(z)}}\right| \nonumber\\
    &\geq \frac{\left|\left(|h_{\epsilon}(z)-1|- |h_{\epsilon}(-d(t), 0) - 1|\right)\right|}{|h_{\epsilon}(z)-1| + |h_{\epsilon}(-d(t), 0)- 1|}.
\end{align}
As $z \in D_t \cap W \setminus {D}_t^{\epsilon}$,
$\operatorname{Re}z_1 \leq - d(t)^{\frac{1}{1 + \epsilon}}$.
Hence,
\begin{align}
    |h_{\epsilon}(z) - 1| &\geq 1 - e^{-c_0^{\epsilon}d(t)^{\frac{1}{(1+\epsilon)^2}}}, \text{ and }\label{A_3}\\
    |1 - h_\epsilon(-d(t), 0)| &\leq d(t)^{\frac{1}{1+ \epsilon}},\label{A_4}
\end{align}
where $c_0^{\epsilon} = \operatorname{cos}\left(\frac{\pi}{2(1 + \epsilon)}\right)$. From \eqref{A_3} and \eqref{A_4}, we get
\begin{align}\label{100}
    {|h_{\epsilon}(z)-1|- |h_{\epsilon}(-d(t), 0) - 1|} \geq  1 - e^{-c^{\epsilon}_0d(t)^{\frac{1}{(1+\epsilon)^2}}} - d(t){^\frac{1}{1+ \epsilon}} > 0,
\end{align}
for all sufficiently small $t > 0$. Therefore, from \eqref{A_2} and \eqref{100}, we get
\begin{align}
     c_{D_t \cap W}^*((-d(t), 0), z)
    &\geq \frac{|h_{\epsilon}(z)-1|- |h_{\epsilon}(-d(t), 0) - 1|}{|h_{\epsilon}(z)-1| + |h_{\epsilon}(-d(t), 0)- 1|}\nonumber\\
    &\geq \frac{ 1 - e^{-c^{\epsilon}_0d(t)^{\frac{1}{(1+\epsilon)^2}}} - d(t)^{\frac{1}{1+ \epsilon}}}{ 1 - e^{-c_0^{\epsilon}d(t)^{\frac{1}{(1+\epsilon)^2}}} + d(t)^{\frac{1}{1+ \epsilon}}} \quad (\text{using \eqref{A_3} and \eqref{A_4}}),\label{101}
\end{align}
for all sufficiently small $t > 0$. Since
\begin{align}
    c_{D_t \cap W}((-d(t), 0), z) &= \operatorname{tanh}^{-1}(c_{D_t \cap W}^{*}((-d(t), 0), z)\nonumber\\
    &\geq \operatorname{tanh}^{-1}\left(\frac{ 1 - e^{-c^{\epsilon}_0d(t)^{\frac{1}{(1+\epsilon)^2}}} - d(t)^{\frac{1}{1+ \epsilon}}}{ 1 - e^{-c_0^{\epsilon}d(t)^{\frac{1}{(1+\epsilon)^2}}} + d(t)^{\frac{1}{1+ \epsilon}}}\right)\quad (\text{from \eqref{101}})\nonumber\\
    &= \frac{1}{2}\operatorname{log}\left\{\frac{1 + \frac{ 1 - e^{-c_0^{\epsilon}d(t)^{\frac{1}{(1+\epsilon)^2}}} - d(t)^{\frac{1}{1+ \epsilon}}}{ 1 - e^{-c_0^{\epsilon}d(t)^{\frac{1}{(1+\epsilon)^2}}} + d(t)^\frac{1}{1+ \epsilon}}}{1 - \frac{ 1 - e^{-c_0^{\epsilon}d(t)^{\frac{1}{(1+\epsilon)^2}}} - d(t)^\frac{1}{1+ \epsilon}}{ 1 - e^{-c_0^{\epsilon}d(t)^{\frac{1}{(1+\epsilon)^2}}} + d(t)^\frac{1}{1+ \epsilon}}}\right\}\nonumber\\
    &= \frac{1}{2}\operatorname{log}\left\{\frac{1 - e^{-c_0^{\epsilon}d(t)^{\frac{1}{(1+\epsilon)^2}}}}{d(t)^{\frac{1}{1+\epsilon}}}\right\}\nonumber,
\end{align}
for each $z \in D_t \cap W \setminus {D}_t^\epsilon$.
Hence
\begin{align}
    c_{D_t \cap W}((-d(t), 0), D_t \cap W \setminus D_t^{\epsilon}) \geq \frac{1}{2}\operatorname{log}\left\{\frac{1 - e^{-c_0^{\epsilon}d(t)^{\frac{1}{(1+\epsilon)^2}}}}{d(t)^{\frac{1}{1+\epsilon}}}\right\},
\end{align}
which implies 
\begin{align}\label{103}
    \operatorname{coth}c_{D_t \cap W}((-d(t), 0), D_t \cap W \setminus D_t^{\epsilon}) \leq \operatorname{coth}\left(\frac{1}{2}\operatorname{log}\left\{\frac{1 - e^{-c_0^{\epsilon}d(t)^{\frac{1}{(1+\epsilon)^2}}}}{d(t)^{\frac{1}{1+\epsilon}}}\right\}\right).
\end{align}
From \eqref{A_1} and \eqref{103}, we have
\begin{align}\label{104}
    \limsup_{t \to 0^{+}}\frac{M_{{D}_t^{\epsilon}}^K((-d(t), 0);R^2_t \circ R^1_t(\xi))}{M_{D_t \cap W}^K((-d(t), 0);R^2_t \circ R^1_t(\xi))} \leq 1.
\end{align}
Since Kobayashi metric is monotonically decreasing with respect to the domain and $D_t^{\epsilon} \subset D_t \cap W$, 
\begin{align}\label{105}
    \liminf_{t \to 0^{+}}\frac{M_{{D}_t^{\epsilon}}^K((-d(t), 0);R^2_t \circ R^1_t(\xi))}{M_{D_t \cap W}^K((-d(t), 0);R^2_t \circ R^1_t(\xi))} \geq 1.
\end{align}
By using \eqref{104} and \eqref{105}, we have
\begin{align}\label{loc of kobayashi}
    \lim_{t \to 0^+} \frac{M_{{D}_t^{\epsilon}}^K((-d(t), 0);R^2_t \circ R^1_t(\xi))}{M_{D_t \cap W}^K((-d(t), 0);R^2_t \circ R^1_t(\xi))} = 1.
\end{align}
Therefore, by using Lemma \ref{localization of Kobayashi}, \eqref{94} and \eqref{loc of kobayashi}, we have
\begin{align}\label{60}
    \lim_{t \to 0^{+}} \frac{M^K_D(q(t); \xi)}{M^K_{D^{\epsilon}_t}((-d(t), 0); R^2_t \circ R^1_t(\xi))} = 1,
\end{align}
for each $\epsilon > 0$. 

\textbf{Step $2$.}
Let $\xi \in \mathbb{C}^{n + 1} \setminus \{0\}$. We denote the complex normal and complex tangential components of $\xi$ with respect to $\pi(q(t))$ as $\xi_{N,t}$ and $\xi_{T,t}$, respectively. This results in the decomposition $\xi = \xi_{N,t} + \xi_{T,t}$, where $\xi_{T,t} \in T_{\pi(q(t))}^{\mathbb{C}}(bD)$ and $\xi_{N,t} \perp T_{\pi(q(t))}^{\mathbb{C}}(bD)$.

Expressing this in detail:
\begin{align}
     \xi_{N, t} &= \langle R^2_t \circ R^1_t(\xi), e_1\rangle (R_t^2 \circ R_t^1)^{-1}(e_1), \text{ and}\\
     \xi_{T,t} &= \sum_{j = 2}^{n + 1}\langle R^2_t \circ R^1_t(\xi), e_j\rangle (R_t^2 \circ R_t^1)^{-1}(e_j).
\end{align}
Furthermore, we have $|\xi_{N, t}| = |\langle R^2_t \circ R^1_t(\xi), e_1 \rangle|$, and $|\xi_{T,t}| = \sqrt{\sum_{j = 2}^{n + 1}|\langle R^2_t \circ R^1_t(\xi), e_j\rangle|^2}$.

By using Lemma \ref{ScalingLemma2}, 
for $\epsilon, \delta > 0$, there exists $t_0(\epsilon, \delta)>0$, such that
\begin{align}\label{61}
     (1 - \delta)(\mathbb{D}\times B_n(0,1)) \subset f\circ\Sigma\left({D}_t^{\epsilon}\right) \subset \mathbb{D} \times B_n(0, {d}_2^\epsilon(t)/d^*(t)),
\end{align}
for each $0 < t< t_0(\epsilon, \delta)$. 
Since
\begin{align}\label{62}
    M^K_{D_t^{\epsilon}}((-d(t), 0); R_t^2 \circ R_t^1(\xi)) = M_{f \circ \Sigma(D_t^{\epsilon})}^K(0;(f \circ \Sigma)'(-d(t), 0)R_t^2 \circ R_t^1(\xi))
\end{align}
By using \eqref{61}, we have
\begin{align}\label{63}
    M_{f \circ \Sigma(D_t^{\epsilon})}^K(0;(f \circ \Sigma)'(-d(t), 0)R_t^2 \circ R_t^1(\xi)) &\leq M^K_{(1 - \delta)(\mathbb{D} \times B_n(0, 1))}(0; (f \circ \Sigma)'(-d(t), 0)R_t^2 \circ R_t^1(\xi))\nonumber\\
    &= M_{\mathbb{D} \times B_n(0, 1)}^K\left(0; \frac{(f \circ \Sigma)'(-d(t), 0)R_t^2 \circ R_t^1(\xi)}{(1 - \delta)}\right)\nonumber\\
    &= \frac{1}{(1 - \delta)} M_{\mathbb{D} \times B_n(0, 1)}^K\left(0; (f \circ \Sigma)'(-d(t), 0)R_t^2 \circ R_t^1(\xi)\right).
\end{align}
Here
\begin{align}\label{64}
    (f \circ \Sigma)'(-d(t), 0) R_t^2 \circ R_t^1(\xi) = \langle R_t^2 \circ R_t^1 (\xi), e_1\rangle \frac{e_1}{2d(t)} + \sum_{j = 2}^{n + 1}\langle R_t^2 \circ R_t^1 (\xi), e_j\rangle \frac{e_j}{d^{*}(t)}.
\end{align}
From \eqref{62}, \eqref{63}, \eqref{64} and Lemma \ref{Kobayashi formula}, we have
\begin{align}
    \frac{M_{D_t^{\epsilon}}^K((-d(t), 0); R_t^2 \circ R_t^1(\xi))}{\operatorname{max}\left\{{|\xi_{N, t}|}/{2d(t)}, {|\xi_{T, t}|}/{d^{*}(t)}\right\}} \leq \frac{1}{1 - \delta},
\end{align}
for each $0 < t < t_0(\epsilon, \delta)$.
Hence
\begin{align}\label{66}
    \limsup_{t \to 0^{+}} \frac{M_{D_t^{\epsilon}}^K((-d(t), 0); R_t^2 \circ R_t^1(\xi))}{\operatorname{max}\left\{{|\xi_{N, t}|}/{2d(t)}, {|\xi_{T, t}|}/{d^{*}(t)}\right\}} \leq 1.
\end{align}
By using \eqref{61} and the inequality $d_2^{\epsilon}(t) \geq d^{*}(t)$ for each $0 < t < t_0(\epsilon, \delta)$ (follows from the definitions of $d_2^{\epsilon}$ and $d^{*}$), we also have 
\begin{align}\label{67}
    M_{f \circ \Sigma(D_t^{\epsilon})}^K(0;(f \circ \Sigma)'(-d(t), 0)R_t^2 \circ R_t^1(\xi)) &\geq M^K_{\mathbb{D} \times B_n(0, d_2^{\epsilon}(t)/d^{*}(t))}(0; (f \circ \Sigma)'(-d(t), 0)R_t^2 \circ R_t^1(\xi))\nonumber\\
    &\geq M^K_{\frac{d_2^{\epsilon}(t)}{d^{*}(t)}(\mathbb{D} \times B_n(0, 1))}(0; (f \circ \Sigma)'(-d(t), 0)R_t^2 \circ R_t^1(\xi)) \nonumber  \\
    &= \frac{d^{*}(t)}{d_2^{\epsilon}(t)}M_{\mathbb{D} \times B_n(0, 1)}^K\left(0; (f \circ \Sigma)'(-d(t), 0)R_t^2 \circ R_t^1(\xi)\right).
\end{align}
Again from \eqref{62}, \eqref{67}, \eqref{64} and Lemma \ref{Kobayashi formula}, we have
\begin{align}
    \frac{M_{D_t^{\epsilon}}^K((-d(t), 0); R_t^2 \circ R_t^1(\xi))}{\operatorname{max}\left\{\frac{|\xi_{N, t}|}{2d(t)}, \frac{|\xi_{T, t}|}{d^{*}(t)}\right\}} \geq \frac{d^{*}(t)}{d^{\epsilon}_2(t)},
\end{align}
for each $0 < t < t_0(\epsilon, \delta)$. We now use Lemma \ref{3.1}, we get 
\begin{align}\label{69}
    \liminf_{t \to 0^{+}}\frac{M_{D_t^{\epsilon}}^K((-d(t), 0); R_t^2 \circ R_t^1(\xi))}{\operatorname{max}\left\{\frac{|\xi_{N, t}|}{2d(t)}, \frac{|\xi_{T, t}|}{d^{*}(t)}\right\}} \geq \frac{1}{(1 + \epsilon)^{\frac{1}{m}}},
\end{align}
for each $\epsilon > 0$. 
Consider
\begin{align}\label{70}
    \frac{M_{D}^{K}(q(t); \xi)}{\operatorname{max}\{|\xi_{N, t}|/2d(t), |\xi_{T, t}|/d^{*}(t)\}} = \frac{M^K_D(q(t); \xi)}{M^K_{D^{\epsilon}_t}((-d(t), 0); R^2_t \circ R^1_t(\xi))} \frac{{M^K_{D^{\epsilon}_t}((-d(t), 0); R^2_t \circ R^1_t(\xi))}}{\operatorname{max}\{|\xi_{N, t}|/2d(t), |\xi_{T, t}|/d^{*}(t)\}}.
\end{align}
From \eqref{60}, \eqref{66} and \eqref{70}, we get
\begin{align}\label{71}
    \limsup_{t \to 0^{+}}\frac{M_{D}^{K}(q(t); \xi)}{\operatorname{max}\{|\xi_{N, t}|/2d(t), |\xi_{T, t}|/d^{*}(t)\}} \leq 1.
\end{align}
Again from \eqref{60}, \eqref{69} and \eqref{70}, we get
\begin{align}
    \liminf_{t \to 0^{+}}\frac{M_{D}^{K}(q(t); \xi)}{\operatorname{max}\{|\xi_{N, t}|/2d(t), |\xi_{T, t}|/d^{*}(t)\}} \geq \frac{1}{(1 + \epsilon)^{\frac{1}{m}}},
\end{align}
for each $\epsilon > 0$. Hence
\begin{align}\label{73}
     \liminf_{t \to 0^{+}}\frac{M_{D}^{K}(q(t); \xi)}{\operatorname{max}\{|\xi_{N, t}|/2d(t), |\xi_{T, t}|/d^{*}(t)\}} \geq 1.
\end{align}
Therefore, from \eqref{71} and \eqref{73}, we get
\begin{align*}
    \lim_{t \to 0^+} \frac{M^{K}_{D}(q(t); \xi)}{\operatorname{max}\left\{\frac{|\xi_{N, t}|}{2d(t)},\frac{|\xi_{T,t}|}{d^*(t)}\right\}} = 1,
\end{align*}
for $\xi \in \mathbb{C}^{n + 1} \setminus \{0\}$.
\end{proof}
\section*{Acknowledgements}
This work is completed during my PhD studies under the guidance of my advisor, Sivaguru Ravisankar. I am deeply grateful to him for his numerous insightful discussions, valuable suggestions, and constructive comments, all of which have greatly enhanced the quality of this work.
\def\MR#1{\relax\ifhmode\unskip\spacefactor3000 \space\fi%
  \href{http://www.ams.org/mathscinet-getitem?mr=#1}{MR#1}}
 
\end{document}